\documentclass{my-aims}

\usepackage{amsmath}
\usepackage{paralist}
\usepackage{graphics} 
\usepackage{epsfig}   
\usepackage{epstopdf} 
\usepackage[colorlinks=true]{hyperref}

\usepackage{cite}

\allowdisplaybreaks

\hypersetup{urlcolor=blue, citecolor=red}
\usepackage{hyperref}

\def\ppages{X--XX}

\newtheorem{theorem}{Theorem}[section]

\newtheorem{assump}[theorem]{Assumption}

\theoremstyle{definition}
\newtheorem{definition}[theorem]{Definition}
\newtheorem{remark}{Remark}
\newtheorem*{notation}{Notation}

\title[A sufficient optimality condition]{A sufficient optimality 
condition for delayed state-linear optimal control problems}

\author[A. P. Lemos-Pai\~{a}o, C. J. Silva and D. F. M. Torres]{}



\email{anapaiao@ua.pt}
\email{cjoaosilva@ua.pt}
\email{delfim@ua.pt}

\thanks{This work is part of first author's Ph.D., 
which is carried out at the University of Aveiro.}

\thanks{$^*$Corresponding author: Delfim F. M. Torres (delfim@ua.pt)}

\begin{document}

\maketitle

\centerline{\scshape Ana P. Lemos-Pai\~{a}o, Cristiana J.~Silva and Delfim F.~M.~Torres$^*$}
\medskip
{\footnotesize
\centerline{Center for Research and Development in Mathematics and Applications (CIDMA)}
\centerline{Department of Mathematics, University of Aveiro, 3810-193 Aveiro, Portugal}}

\bigskip

% -------------------------------

\begin{abstract}
We give answer to an open question by proving 
a sufficient optimality condition for state-linear optimal control problems 
with time delays in state and control variables.
In the proof of our main result, we transform a delayed state-linear optimal 
control problem to an equivalent non-delayed problem. 
This allows us to use a well-known theorem that ensures a sufficient optimality 
condition for non-delayed state-linear optimal control problems.  
An example is given in order to illustrate the obtained result.
\end{abstract}

% -------------------------------

\section{Introduction}

Time delays occur in many dynamical systems such as biological, 
chemical, mechanical and economical systems (see, e.g., 
\cite{Bashier:2017,Elaiw,Gollmann2,Klamka:2016,Santos,Stumpf,Xia:2009,Xu2,Xu}). 
Dynamic systems with time delays, in both state and control variables, play 
an important role in the modelling of real-life phenomena in various fields 
of applications \cite{Gollmann,Gollmann2}. For instance, in \cite{Rocha} 
the incubation and pharmacological delays are modelled through the introduction 
of time delays in both state and control variables. In \cite{Silva},
Silva, Maurer and Torres introduce time delays in the state and control variables 
for tuberculosis modelling. They represent the time delay on the diagnosis 
and commencement of treatment of individuals with active tuberculosis infection 
and the delays on the treatment of persistent latent individuals, due to clinical 
and patient reasons. There is a vast literature on delayed optimal control problems, 
also called retarded, time-lag, or hereditary optimal control problems. See, e.g.,
\cite{Banks,Boccia1,Friedman,Gollmann,Halanay,Oguztoreli} and references cited therein.

Delayed linear differential systems have also been investigated, their importance 
being recognized both from a theoretical and practical points of view. For instance, 
in \cite{Friedman} Friedman considers linear hereditary processes and apply to them
Pontryagin's method, deriving necessary optimality conditions as well as existence and 
uniqueness results. Analogously, in \cite{Oguztoreli} linear delayed 
differential equations and optimal control problems involving this kind of systems 
are studied. Since these first works, many researchers have devoted their attention 
to linear quadratic optimal control problems with time delays, see, e.g., \cite{Cacace:SCL:2016,Delfour:SIAM,Eller:TAC:1970,Khellat:JOTA:2009,Palanisamy:IEEE:1983}. 
It turns out that for linear quadratic delayed optimal control problems 
it is possible to provide an explicit formula for the optimal controls 
\cite{Cacace:SCL:2016,Khellat:JOTA:2009,Palanisamy:IEEE:1983}.

Optimal control problems with a differential system that is 
linear both in state and control variables have been studied in
\cite{Cacace:SCL:2016,Chyung,Delfour:SIAM,Eller:TAC:1970,Khellat:JOTA:2009,%
Koepcke,Koivo,Lee1,Oguztoreli_2,Palanisamy:IEEE:1983}. 
In \cite{Delfour:SIAM,Koepcke,Palanisamy:IEEE:1983}, 
the system is delayed with respect to state and control variables. 
In \cite{Chyung,Oguztoreli_2}, the system only considers delays in the state variable. 
Chyung and Lee derive necessary and sufficient optimality conditions in \cite{Chyung} 
while O\v{g}uzt\"{o}reli only proves necessary conditions \cite{Oguztoreli_2}. 
Certain necessary conditions analysed by Chyung and Lee in \cite{Chyung} 
have been already derived in \cite{Kharatishvili_2,Pontryagin,Popov}. However,
the system considered in \cite{Chyung} is different from the previously studied
hereditary systems, which do not require a initial function of state. 
In \cite{Eller:TAC:1970}, Eller et al. derive a sufficient condition for a control 
to be optimal for certain problems with time delay. The problems studied 
by Eller et al. and Khellat, respectively in \cite{Eller:TAC:1970} and \cite{Khellat:JOTA:2009},  
consider only one constant lag in the state. The research done by Lee in \cite{Lee1} 
is different from ours, because in \cite{Lee1} the aim is to minimize a cost functional, 
which does not consider delays, subject to a differential system that is linear 
in state and control variables, and to another constraint. In their differential 
system, the state variable depends on a constant and fixed delay and the control 
variable depends on a constant lag, which is not specified a priori. Note that the 
differential system of the problem considered in \cite{Koivo} is similar 
to the one of \cite{Lee1}. Although Banks has studied non-linear delayed 
problems without lags in the control, he has also analyzed problems that 
are linear and delayed with respect to control \cite{Banks}. Recently, 
Cacace et al. studied optimal control problems that involve linear differential 
systems with variable delays only in the control \cite{Cacace:SCL:2016}. 
The problems analyzed in the present paper are different from those considered 
in the mentioned works, because here the problems involve differential 
systems that are linear with respect to state, but not with respect to the control. 
Furthermore, we consider a constant lag in the state and another one in the control. 
These two delays are in general not equal.

In \cite{Hughes}, Hughes firstly consider variational problems with only 
one constant lag and derive various necessary and a sufficient optimality 
condition for them. The variational problems in \cite{Hughes} can easily 
be transformed to control problems with only one constant delay (see, e.g., 
\cite[p.~53--54]{Paiao}). Hughes also investigate an optimality condition for 
a control problem with a constant delay, which is the same for state and control. 
Therefore, the problems investigated in \cite{Hughes} are different from the problems 
studied by us, because in the present paper the delay of state is not necessarily 
equal to the delay of control. The problems analyzed by Chan and Yung \cite{Chan}
and by Sabbagh \cite{Sabbagh} are similar to the first problems studied 
by Hughes in \cite{Hughes}. So, for the same reason, the problems investigated 
in \cite{Chan,Sabbagh} are different from ours. The problems considered in 
\cite{Hughes,Sabbagh} are also considered in \cite{Palm} 
by Palm and Schmitendorf. For such problems, they derive two conjugate-point conditions, 
which are not equivalent. Note that their conditions are only necessary  
and do not give a set of sufficient conditions \cite{Palm}.

In \cite{Jacobs}, Jacobs and Kao investigate delayed problems 
that consist to minimize a cost functional without delays subject 
to a differential system defined by a non-linear function 
with a delay in state and another one in the control. Similar to
our case, these delays do not have to be equal. In contrast,
our cost functional contains also time delays, therefore
being more general than the one considered in \cite{Jacobs}.
Jacobs and Kao transform the problem using a Lagrange-multiplier technique 
and prove a regularity result in the form of a controllability condition, 
as well as some necessary optimality conditions. Then, in some special
restricted cases, they prove existence, uniqueness and sufficient conditions. 
Such restricted problems consider a differential system that is linear 
in state and in control variables. Thus, the sufficient conditions 
of \cite{Jacobs} are derived for problems that are less general than ours.

The delayed optimal control problems analyzed by Schmitendorf
in \cite{Schmitendorf} have a cost functional and a differential 
system that are more general than ours. However, in \cite{Schmitendorf} 
the control takes its values in all $\mathbb{R}^m$ while in the present paper 
the control values belong to a set $\Omega\subset\mathbb{R}^m$, 
$m\in\mathbb{N}$. In \cite{Lee3}, Lee and Yung study a problem 
that is similar to the one considered in \cite{Schmitendorf}, 
where the control belongs to a subset of $\mathbb{R}^m$, as we consider here. 
First and second-order sufficient conditions are shown in \cite{Lee3}. 
Nevertheless, the conditions of \cite{Lee3} are not constructive 
and practical for the computation of the optimal solution. 
Indeed, as hypothesis, it is assumed existence of a symmetric matrix 
under some conditions, for which is not given a method to calculate 
its expression. Another similar problem to our is studied by Bokov 
in \cite{Bokov}, in order to arise a necessary optimality condition 
in an explicit form. Moreover, a solution to the problem with 
infinite time horizon is given in \cite{Bokov}. In contrast, 
in the present paper we are interested to derive sufficient optimality conditions. 
As it is well known, and as Hwang and Bien write in \cite{Hwang}, 
many investigations have directed their efforts to seek sufficient conditions 
for control problems with delays: see, e.g., 
\cite{Chyung,Eller:TAC:1970,Hughes,Jacobs,Lee3,Schmitendorf}. 
In \cite{Hwang}, Hwang and Bien prove a sufficient condition for problems 
involving a differential affine time-delay system with 
the same lag for the state and the control. Thus, the differential 
system considered in the present article is obviously more general. 
In 1996, Lee and Yung derived various first and second-order sufficient 
conditions for non-linear optimal control problems, with only a constant 
delay in the state, and considering functions that do not have to be convex 
\cite{Lee2}. As in \cite{Chan,Lee3}, second-order sufficient 
conditions are shown to be related to the existence of solutions 
of a Riccati-type matrix differential inequality.

Optimal control problems with multiple delays have also been investigated. 
In \cite{Halanay}, Halanay derive necessary conditions for some optimal 
control problems with various time lags in state and control variables, 
using the abstract multiplier rule of Hestenes \cite{Hestenes}. 
In \cite{Halanay}, all delays related to state are equal to each other 
and the same happens with the delays associated to the control. 
Note that the results of \cite{Friedman,Kharatishvili_2} are obtained 
as particular cases of problems considered in \cite{Halanay}. Later, in 1973, 
a necessary condition is derived for an optimal control problem 
that involves multiple constant lags only in the control. This delayed 
dependence occurs both in the cost functional and in the differential 
system, which is defined by a non-linear function \cite{Soliman}. 
In \cite{Kharatishvili_3}, Kharatishvili and Tadumadze prove the existence 
of an optimal solution and a necessary condition for optimal control systems 
with multiple variable time lags in the state and multiple variable 
commensurable time delays in the control. Later, an optimal control 
problem where the state variable is solution of an integral equation 
with multiple delays, both for state and control variables, 
is studied by Bakke in \cite{Bakke}. Furthermore, necessary 
conditions and Hamilton--Jacobi equations are derived. 
In 2013, Boccia, Falugi, Maurer and Vinter derived necessary conditions 
for a free end-time optimal control problem subject to a non-linear 
differential system with multiple delays in the state \cite{Boccia1}. 
The control variable is not influenced by time lags in \cite{Boccia1}. 
Recently, in 2017, Boccia and Vinter obtained necessary conditions for 
a fixed end-time problem with a constant and unique delay for all variables,
as well as free end-time problems without control delays \cite{Boccia2}.

As Guinn wrote, the classical methods of obtaining necessary conditions 
for retarded optimal control problems (used, for instance, by Halanay 
in \cite{Halanay}, Kharatishvili in \cite{Kharatishvili} 
and O\v{g}uzt\"{o}reli in \cite{Oguztoreli}) require complicated 
and extensive proofs \cite{Guinn} (see, e.g., 
\cite{Banks,Friedman,Halanay,Kharatishvili,Oguztoreli}). In $1976$, 
Guinn proposed a method whereby we can reduce some specific time-lag 
optimal control problems to equivalent and augmented optimal control 
problems without delays \cite{Guinn}. By reducing delayed optimal control 
problems into non-delayed ones, we can then use well-known theorems 
applicable for optimal control problems without delays to derive 
desired optimality conditions for delayed problems \cite{Guinn}. 
In \cite{Guinn}, Guinn study specific optimal control problems with a constant 
delay in state and control variables. These two delays are equal. 
Later, in $2009$, G\"{o}llmann, Kern and Maurer studied optimal control problems 
with a constant delay in state and control variables subject to 
mixed control-state inequality constraints \cite{Gollmann}. In that 
research, the delays do not have to be equal. For technical reasons, 
the authors need to assume that the ratio between these two time delays 
is a rational number \cite{Gollmann}. In \cite{Gollmann}, the method used 
by Guinn in \cite{Guinn} is generalized and, consequently, a non-delayed optimal 
control problem is obtained again. Pontryagin's Minimum Principle, 
for non-delayed control problems with mixed state-control constraints, 
is used and first-order necessary optimality conditions are derived 
for retarded problems \cite{Gollmann}. Furthermore, G\"{o}llmann, 
Kern and Maurer discuss the Euler discretization for the retarded 
problem and some analytical examples versus correspondent numerical 
solutions are given. Later, in 2014, G\"{o}llmann and Maurer generalized 
the research mentioned before, by studying optimal control problems 
with multiple and constant time delays in state and control, 
involving mixed state-control inequality constraints \cite{Gollmann2}. 
Again, necessary optimality conditions are derived \cite{Gollmann2}. 
Note that the works \cite{Gollmann,Gollmann2,Guinn,Halanay} 
consider non-linear delayed differential systems.

In the present paper, we consider optimal control problems 
that consist to minimize a delayed non-linear cost functional 
subject to a delayed differential system that is linear 
with respect to state, but not with respect to control. 
The delay in the state is the same for the cost functional 
and for the differential system. The same happens with 
the time lag of the control variable. We derive 
a sufficient optimality condition for this type of problems. 
Note that the cost functional does not have to be quadratic, 
but it satisfies some continuity and convexity assumptions. 
To the best of our knowledge, this gives answer to an open question. 
Note that the constant delays on the state and control variables 
do not have to be equal, but we ensure the commensurability assumption 
between state and control delays, similarly
to G\"{o}llmann, Kern and Maurer in \cite{Gollmann}. 
Indeed, we follow the approach  of \cite{Gollmann} 
and Guinn \cite{Guinn}, that is, we transform 
the delayed optimal control problem into an equivalent 
non-delayed optimal control problem and then apply a classical 
sufficient optimality condition \cite[p.~340--343]{Lee}.

The paper is organized as follows. In Section~\ref{sect_preliminaries}, 
we define the optimal control problem without delays for which 
the sufficient optimality condition \cite[p.~340--343]{Lee} holds. 
In Section~\ref{sec:delayOC}, we define our retarded optimal control 
problem with constant time delays in state and control variables. 
Then, in Section~\ref{sect_mainres}, we prove a sufficient optimality 
condition for the problem stated in Section~\ref{sec:delayOC}. 
A concrete example is solved in detail in Section~\ref{sect_example},
with the purpose to illustrate our main result. We end with some 
conclusions in Section~\ref{sec:conclusion}.

% -------------------------------

\section{Non-delayed state-linear optimal control problem}
\label{sect_preliminaries}

We begin by defining a non-delayed state-linear optimal control problem 
and recall a well-known sufficient optimality result 
for such class of problems.

Consider the non-delayed state-linear optimal control problem (LP) which consists to
\begin{equation}
\label{eq:linear:cost}
\min\ \ C[u]=\int_{a}^{b}f^0(t,x(t))+g^0(t,u(t))dt
\end{equation}
subject to the control system in $\mathbb{R}^n$
\begin{equation}
\label{eq:linear:CS}
\dot{x}(t)=A(t)x(t)+g(t,u(t))
\end{equation}
with initial boundary condition
\begin{equation}
\label{eq:InitCond}
x(a) = x_a
\end{equation}
and final boundary condition $x(b)\in\Pi$; 
where $\Pi\subseteq\mathbb{R}^n$ is a closed convex set,
$x(t)\in\mathbb{R}^n$, $u(t)\in\Omega\subseteq\mathbb{R}^m$ 
and $A(t)$ is a real $n \times n$ matrix, $t \in [a, b]$. 
Functions $f^0$, $\partial_2 f^0$ and $g^0$ are assumed to be 
continuous for all $(t,x,u)\in[a,b]\times\mathbb{R}^{n+m}$.

\begin{notation}
Along the text we use the notation $\partial_if$ to denote the partial 
derivative of a certain function $f$ with respect to its $i$th argument. 
For example, $\displaystyle\partial_2f^0=\frac{\partial f^0}{\partial x}$.
\end{notation}

\begin{definition}
An admissible process to (LP) is given by a pair of functions 
$(x,u)\in W^{1,\infty}([a,b],\mathbb{R}^n)\times L^{\infty}([a,b],\mathbb{R}^m)$ 
that satisfies conditions \eqref{eq:linear:CS} and \eqref{eq:InitCond}.
\end{definition}

The following theorem gives a sufficient optimality condition for problem (LP).

\begin{theorem}[See Theorem~5, Section~5.2 of \cite{Lee}] 
\label{theo_suf_linear}
Consider problem (LP) and assume that
\begin{enumerate}
\item functions $f^0$, $\partial_2 f^0$, $g^0$, $A$ and $g$ 
are continuous for all $(t,x,u)\in[a,b]\times\mathbb{R}^{n+m}$;

\item $f^0(t,x)$ is a convex function in $x$ for each fixed $t\in[a,b]$;

\item for almost all $t\in[a,b]$, $u^*$ is a control with response 
$x^*$ that satisfies the \emph{maximality condition}
\begin{equation*}
H(t,x(t),u^*(t),\eta(t))=\max_{u\in\Omega}H(t,x(t),u,\eta(t)),
\end{equation*}
where
\begin{equation*}
H(t,x,u,\eta)=-[f^0(t,x)+g^0(t,u)]+\eta [A(t)x+g(t,u)],
\end{equation*}
and $\eta(t)$ is any nontrivial solution of the \emph{adjoint system}
\begin{equation*}
\dot{\eta}(t)=\partial_2 f^0(t,x^*(t))-\eta(t) A(t),
\end{equation*}
satisfying the \emph{transversality condition} that ensures that $\eta(b)$ 
is an inward normal vector of $\Pi$ at the boundary point $x^*(b)$.
\end{enumerate}
Then, $u^*$ is an optimal control that leads 
to the minimal cost $C[u^*]$.
\end{theorem}

\begin{remark}
\label{vacuous}
Note that if $\Pi=\{x_b\}$, then the transversality condition 
of Theorem~\ref{theo_suf_linear} is vacuous, because $\Pi$ has a single point.
If $\Pi=\mathbb{R}^n$, then $\eta(b)=[
\begin{matrix}
0 & \cdots & 0
\end{matrix}
]_{1\times n}$.
\end{remark}

% -------------------------------

\section{Delayed state-linear optimal control problem}
\label{sec:delayOC}

In this paper we are interested in state-linear optimal control problems 
with discrete time delays $r \geq 0$ in the state variables 
$x(t) \in \mathbb{R}^n$ and $s \geq 0$ in the control variables 
$u(t) \in \mathbb{R}^m$, $(r, s) \neq (0, 0)$. The delayed 
state-linear optimal control problem ($LP_D$) consists in
\begin{equation*}
\min\ \ C_D[u]=\int_{a}^{b}f^0(t,x(t),x(t-r))+g^0(t,u(t),u(t-s))dt
\end{equation*}
subject to the delayed differential system
\begin{equation}
\label{eq:Delay:linear:CS}
\dot{x}(t)=A(t)x(t)+A_D(t)x(t-r)+g(t,u(t))+g_D(t,u(t-s))
\end{equation}
with the following initial functions
\begin{equation}
\label{eq:InitCond:delayControl}
\begin{split}
x(t)&=\varphi(t),\ t\in[a-r,a],\\
u(t)&=\psi(t),\ t\in[a-s,a[,
\end{split}
\end{equation}
where $x(t)\in\mathbb{R}^n$ for each $t\in[a-r,b]$ 
and $u(t)\in\Omega\subseteq\mathbb{R}^m$ for each $t\in[a-s,b]$.

\begin{definition}
An admissible process to problem ($LP_D$) is given by a pair 
of functions $(x,u)\in W^{1,\infty}([a-r,b],\mathbb{R}^n)
\times L^{\infty}([a-s,b],\mathbb{R}^m)$ that satisfies 
conditions \eqref{eq:Delay:linear:CS}--\eqref{eq:InitCond:delayControl}.
\end{definition}

% ------------------------------- 

\section{Main Result}
\label{sect_mainres}

In what follows, we assume that the time delays $r$ and $s$ 
respect the following commensurability assumption.

\begin{assump}[Commensurability assumption]
\label{assump}
We consider $r,s\geq0$ not simultaneously equal to zero 
and commensurable, that is, 
$$
(r,s)\neq(0,0)
$$ 
and
\begin{equation*}
\frac{r}{s}\in\mathbb{Q}\ \text{ for }\ s>0
\ \text{ or }\ \frac{s}{r}\in\mathbb{Q}\ \text{ for }\ r>0.
\end{equation*}
\end{assump}
\begin{remark}
The commensurability assumption holds for any couple of rational 
numbers $(r, s)$ for which at least one number is nonzero \cite{Gollmann}.
\end{remark}

\begin{theorem}
\label{theo_delay}
Consider problem ($LP_D$) and assume that
\begin{enumerate}
\item functions $f^0$, $\partial_2f^0$, $\partial_3f^0$, $g^0$, 
$g$, $g_D$, $A$ and $A_D$ are continuous for all their arguments;

\item $f^0(t,x,x_r)$ is a convex function in $(x,x_r)\in\mathbb{R}^{2n}$ 
for each $t\in[a,b]$;

\item for almost all $t\in[a,b]$, $u^*$ is a control with response $x^*$ 
that satisfies the \emph{maximality condition}
\begin{equation}
\label{max_cond_delayed_theo}
\begin{split}
  H_D^1&(t,x(t),x(t-r),u^*(t),u^*(t-s),\eta(t))\\
 &+H_D^0(t+s,x(t+s),x(t+s-r),u^*(t+s),u^*(t),\eta(t+s))\chi_{[a,b-s]}(t)\\
=&\max_{u\in\Omega}\{H_D^1(t,x(t),x(t-r),u,u^*(t-s),\eta(t))\\
&+H_D^0(t+s,x(t+s),x(t+s-r),u^*(t+s),u,\eta(t+s))\chi_{[a,b-s]}(t)\},
\end{split}
\end{equation}
where
\begin{equation*}
\begin{split}
H_D^p(t,x,y,u,v,\eta)=&-[f^0(t,x,y)+g^0(t,u,v)]\\
&+\eta \left[A(t)x+A_D(t)y+pg(t,u)+(1-p)g_D(t,v)\right]
\end{split}
\end{equation*}
for $p\in\{0,1\}$, and $\eta(t)$ is any nontrivial solution 
of the \emph{adjoint system}
\begin{equation*}
\begin{split}
\dot{\eta}(t)=&\ \partial_2f^0(t,x^*(t),x^*(t-r))+\partial_3
f^0(t+r,x^*(t+r),x^*(t))\chi_{[a,b-r]}(t)\\
&-\eta(t)A(t)-\eta(t+r)A_D(t+r)\chi_{[a,b-r]}(t)
\end{split}
\end{equation*}
that satisfies the \emph{transversality condition}
$\eta(b)=
[
\begin{matrix}
0 & \cdots & 0
\end{matrix}
]_{1\times n}
$.
\end{enumerate}
Then, $u^*$ is an optimal control that leads to the minimal cost $C_D[u^*]$.
\end{theorem}

\begin{proof}
We transform the delayed state-linear optimal control problem ($LP_D$) 
into an equivalent non-delayed state-linear optimal control (LP) type problem,  
using the approach of \cite{Gollmann,Guinn}, and then we apply 
Theorem~\ref{theo_suf_linear}. Without loss of generality, 
we assume the first case of Assumption~\ref{assump}, that is,
$\displaystyle\frac{r}{s}\in\mathbb{Q}$ for $r>0$ and $s>0$.
Consequently, there exist $k,l\in\mathbb{N}$ such that
\begin{equation*}
\frac{r}{s}=\frac{k}{l}
\Leftrightarrow
rl=sk
\Leftrightarrow
\frac{r}{k}=\frac{s}{l}.
\end{equation*}
Thus, let us divide the interval $[a,b]$ into $N$ subintervals 
of amplitude $h:=\displaystyle\frac{r}{k}=\frac{s}{l}$.
We can note that
\begin{equation*}
r=hk\ \text{ and }\ s=hl.
\end{equation*}
Furthermore, we also assume that 
\begin{equation}
\label{hip_h}
a+hN=b\ \text{ and }\ N>2k+1,
\end{equation}
with $N\in\mathbb{N}$.

\begin{remark}
If $b-a$ is not a multiple of $h$ ($b-a\neq hN$), then we can study problem 
($LP_D$) for $t\in[a,\tilde{b}]$, where $\tilde{b}$ is the smallest multiple 
of $h$, which is greater than $b$. Thus, we also study problem ($LP_D$) 
for $t\in[a,b]$, because $b<\tilde{b}$.
\end{remark}

For $i=0,\ldots,N-1$ and for $t\in[a,a+h]$, 
we define new variables
\begin{equation*}
\xi_i(t)=x(t+hi)\ \text{ and }\ \theta_i(t)=u(t+hi).
\end{equation*}
\begin{figure}[htp]
\begin{center}
\includegraphics[scale=0.8]{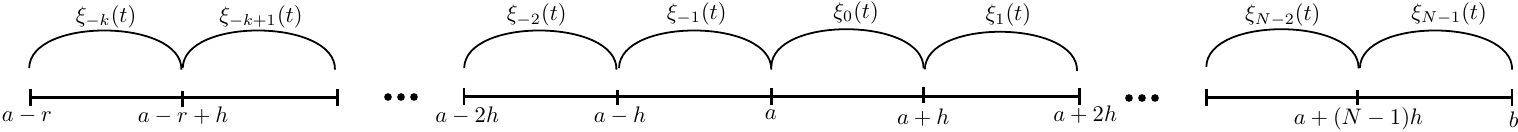}
\caption{Scheme of new state variables.}
\label{esquema}
\end{center}
\end{figure}
In Figure~\ref{esquema}, we can observe a simple scheme
for the new state variables. The idea is similar 
for the new control variables. We transform the delayed 
state-linear problem ($LP_D$) into an equivalent non-delayed 
state-linear problem $(\overline{LP})$, which consists to
\begin{equation}
\label{LP-}
\min\ \ \overline{C}[\theta]=\int_{a}^{a+h}
\sum_{i=0}^{N-1}[f^0(t+hi,\xi_i(t),\xi_{i-k}(t))
+g^0(t+hi,\theta_i(t),\theta_{i-l}(t))]dt
\end{equation}
subject to the non-delayed differential system
\begin{equation*}
\dot{\xi_i}(t)=A(t+hi)\xi_i(t)+A_D(t+hi)\xi_{i-k}(t)
+g(t+hi,\theta_i(t))+g_D(t+hi,\theta_{i-l}(t)),
\end{equation*}
$i=0,\ldots,N-1$, and to the initial functions
\begin{equation*}
\begin{split}
& \xi_i(t)=\varphi(t+hi),\ i=-k,\ldots,-1, \quad t\in[a,a+h],\\
& \theta_i(t)=\psi(t+hi),\ i=-l,\ldots,-1,\quad t\in[a,a+h[, \\
& \xi_i(a+h)=\xi_{i+1}(a),\ i=0,\ldots,N-2.
\end{split}
\end{equation*}
Consider that
\begin{equation*}
\begin{split}
&\xi(t)=
\left[\begin{matrix}
\xi_0(t)\\
\xi_1(t)\\
\vdots\\
\xi_{N-1}(t)
\end{matrix}\right],\
\xi^-(t)=
\left[\begin{matrix}
\xi_{-k}(t)\\
\xi_{1-k}(t)\\
\vdots\\
\xi_{-1}(t)
\end{matrix}\right],\
\theta(t)=
\left[\begin{matrix}
\theta_0(t)\\
\theta_1(t)\\
\vdots\\
\theta_{N-1}(t)
\end{matrix}\right]
\text{ and }
\theta^-(t)=
\left[\begin{matrix}
\theta_{-l}(t)\\
\theta_{1-l}(t)\\
\vdots\\
\theta_{-1}(t)
\end{matrix}\right].
\end{split}
\end{equation*}
Observe that the dimensions of $\xi(t)$, $\xi^-(t)$, $\theta(t)$ 
and $\theta^-(t)$ are $Nn\times1$, $kn\times1$, $Nm\times1$ 
and $lm\times1$, respectively. Note also that $\xi$ and $\theta$ 
represent optimization variables and $\xi^-$ and $\theta^-$ not. 
We know, a priori, the expressions of $\xi^-(t)$, $t\in[a,a+h]$, 
and $\theta^-(t)$, $t\in[a,a+h[$. Let us write the objective function expressed 
in \eqref{LP-} as a function of the type presented in \eqref{eq:linear:cost}:
\begin{equation*}
\begin{split}
&\sum_{i=0}^{N-1} f^0(t+hi,\xi_i(t),\xi_{i-k}(t))\\
=& f^0(t,\xi_0(t),\xi_{-k}(t))+f^0(t+h,\xi_1(t),\xi_{1-k}(t))\\
&+\ldots+f^0(t+h(k-1),\xi_{k-1}(t),\xi_{-1}(t))+f^0(t+hk,\xi_k(t),\xi_0(t))\\
&+\ldots+f^0(t+h(N-1),\xi_{N-1}(t),\xi_{N-1-k}(t)).
\end{split}
\end{equation*}
As $\xi_i$, $i=-k,\ldots,-1$, and $h$ are known,
$\sum_{i=0}^{N-1}f^0(t+hi,\xi_i(t),\xi_{i-k}(t))=F^0(t,\xi(t))$.
Similarly, we can write 
$\sum_{i=0}^{N-1}g^0(t+hi,\theta_i(t),\theta_{i-l}(t))=G^0(t,\theta(t))$.
Consequently,
\begin{equation*}
\begin{split}
&\int_{a}^{a+h}\sum_{i=0}^{N-1}[f^0(t+hi,\xi_i(t),\xi_{i-k}(t))
+g^0(t+hi,\theta_i(t),\theta_{i-l}(t))]dt\\
=&\int_{a}^{a+h}[F^0(t,\xi(t))+G^0(t,\theta(t))]dt.
\end{split}
\end{equation*}
In order to apply Theorem~\ref{theo_suf_linear}, 
we have to write the set of constraints
\begin{equation}
\label{eq_state_constraints_PLD}
\dot{\xi_i}(t)=A(t+hi)\xi_i(t)+A_D(t+hi)\xi_{i-k}(t)
+g(t+hi,\theta_i(t))+g_D(t+hi,\theta_{i-l}(t)),
\end{equation}
$i=0,\ldots,N-1$, in the form
\begin{equation}
\label{eq_state_constraint_theo}
\dot{\xi}(t)=\tilde{A}(t)\xi(t)+\tilde{G}(t,\theta(t)).
\end{equation}
For $i=0,\ldots,N-1$, consider that $t_i=t+hi$. Thus, we have
\begin{eqnarray*}
\left[
\begin{matrix}
A(t_0)\xi_0(t)\\
A(t_1)\xi_1(t)\\
\vdots\\
A(t_{N-1})\xi_{N-1}(t)\\
\end{matrix}
\right]_{Nn\times1}
&=& \left[
\begin{matrix}
A(t_0) & \boldsymbol{0} & \cdots & \cdots & \boldsymbol{0}\\
\boldsymbol{0} & A(t_1) & \boldsymbol{0} & \cdots & \boldsymbol{0}\\
\vdots & \ddots & \ddots & \ddots & \vdots\\
\vdots & \ddots & \ddots & \ddots & \boldsymbol{0}\\
\boldsymbol{0} & \cdots & \cdots & \boldsymbol{0} & A(t_{N-1})
\end{matrix}
\right]
\times
\left[\begin{matrix}
\xi_0(t)\\
\xi_1(t)\\
\vdots\\
\xi_{N-1}(t)
\end{matrix}\right]\\
&=& M(t)\xi(t)
\end{eqnarray*}
and
\begin{eqnarray*}
&&\left[
\begin{matrix}
A_D(t_0)\xi_{-k}(t)\\
A_D(t_1)\xi_{1-k}(t)\\
\vdots\\
A_D(t_k)\xi_0(t)\\
\vdots\\
A_D(t_{N-1})\xi_{N-1-k}(t)\\
\end{matrix}
\right]_{Nn\times1}\\
&=&\left[
\begin{matrix}
A_D(t_0) & \boldsymbol{0} & \cdots & \cdots & \cdots & \cdots & \boldsymbol{0}\\
\boldsymbol{0} & A_D(t_1) & \boldsymbol{0} & \cdots & \cdots & \cdots & \boldsymbol{0}\\
\vdots & \ddots & \ddots & \ddots & \ddots & \ddots & \vdots\\
\boldsymbol{0} & \cdots & \boldsymbol{0} & A_D(t_k) & \boldsymbol{0} & \cdots & \boldsymbol{0}\\
\vdots & \ddots & \ddots & \ddots & \ddots & \ddots & \vdots\\
\vdots & \ddots & \ddots & \ddots & \ddots & \ddots & \boldsymbol{0}\\
\boldsymbol{0} & \cdots & \cdots & \cdots & \cdots & \boldsymbol{0} & A_D(t_{N-1})
\end{matrix}
\right]
\times
\left[\begin{matrix}
\xi_{-k}(t)\\
\xi_{1-k}(t)\\
\vdots\\
\xi_0(t)\\
\vdots\\
\xi_{N-1-k}(t)
\end{matrix}\right]\\
&=&\left[
\begin{matrix}
&  &  & \boldsymbol{0}_{kn\times Nn} &  & & & \\
A_D(t_k) & \boldsymbol{0} & \cdots & \cdots & \cdots & \cdots & \cdots & \boldsymbol{0}\\
\boldsymbol{0} & A_D(t_{k+1}) & \boldsymbol{0} & \cdots & \cdots & \cdots & \cdots & \boldsymbol{0} \\
\vdots & \ddots & \ddots & \ddots & \ddots & \ddots & \ddots & \vdots\\
\boldsymbol{0} & \cdots & \boldsymbol{0} & A_D(t_{N-1}) & \boldsymbol{0} & \cdots & \cdots & \boldsymbol{0}\\
\end{matrix}
\right]
\times
\left[\begin{matrix}
\xi_0(t)\\
\vdots\\
\xi_{N-1-k}(t)\\
\vdots\\
\xi_{N-1}(t)
\end{matrix}\right]\\
&&+
\left[
\begin{matrix}
A_D(t_0) & \boldsymbol{0} & \cdots & \cdots & \cdots & \cdots & \boldsymbol{0}\\
\boldsymbol{0} & A_D(t_1) & \boldsymbol{0} & \cdots & \cdots & \cdots & \boldsymbol{0}\\
\vdots & \ddots & \ddots & \ddots & \ddots & \ddots & \vdots\\
\boldsymbol{0} & \cdots & \boldsymbol{0} & A_D(t_{k-1}) & \boldsymbol{0} & \cdots & \boldsymbol{0}\\
&  &  &  \boldsymbol{0}_{(N-k)n\times Nn} &  &  &
\end{matrix}
\right]
\times
\left[\begin{matrix}
\xi_{-k}(t)\\
\vdots\\
\xi_{-1}(t)\\
\boldsymbol{0}_{(N-k)n\times1}
\end{matrix}\right]\\
&=& M_D(t)\xi(t)+M_D^{-}(t)
\left[
\begin{matrix}
\xi^{-}(t)\\
\boldsymbol{0}_{(N-k)n\times1}
\end{matrix}\right].
\end{eqnarray*}
Note that $M(t)$, $M_D(t)$ and $M_D^-(t)$ have dimension 
$Nn\times Nn$. Concluding, we have
\begin{equation*}
\tilde{A}(t)=M(t)+M_D(t).
\end{equation*}
Now, we write the sum of the third and fourth terms of 
\eqref{eq_state_constraints_PLD} as a function of $t$ and $\theta(t)$. 
Thus,
\begin{equation*}
\begin{split}
&\left[
\begin{matrix}
g(t_0,\theta_0(t))+g_D(t_0,\theta_{-l}(t))\\
g(t_1,\theta_1(t))+g_D(t_1,\theta_{1-l}(t))\\
\vdots\\
g(t_{l-1},\theta_{l-1}(t))+g_D(t_{l-1},\theta_{-1}(t))\\
g(t_l,\theta_l(t))+g_D(t_l,\theta_0(t))\\
\vdots\\
g(t_{N-1},\theta_{N-1}(t))+g_D(t_{N-1},\theta_{N-1-l}(t))
\end{matrix}
\right]\\
=&\left[
\begin{matrix}
g(t_0,\theta_0(t))\\
g(t_1,\theta_1(t))\\
\vdots\\
g(t_{l-1},\theta_{l-1}(t))\\
g(t_l,\theta_l(t))+g_D(t_l,\theta_0(t))\\
\vdots\\
g(t_{N-1},\theta_{N-1}(t))+g_D(t_{N-1},\theta_{N-1-l}(t))
\end{matrix}
\right]
+
\left[
\begin{matrix}
g_D(t_0,\theta_{-l}(t))\\
g_D(t_1,\theta_{1-l}(t))\\
\vdots\\
g_D(t_{l-1},\theta_{-1}(t))\\
0\\
\vdots\\
0
\end{matrix}
\right]\\
=&\ g_{\theta}(t,\theta(t))+g_{\theta^{-}}(t,\theta^{-}(t)).
\end{split}
\end{equation*}
As $\xi^{-}(t)$ and $\theta^{-}(t)$ are known, we have that
\begin{equation*}
\tilde{G}(t,\theta(t))=M_D^{-}(t)
\left[
\begin{matrix}
\xi^{-}(t)\\
\boldsymbol{0}_{(N-k)n\times1}
\end{matrix}
\right]
+ g_{\theta}(t,\theta(t))+g_{\theta^{-}}(t,\theta^{-}(t)).
\end{equation*}
Therefore, we have the set of constraints \eqref{eq_state_constraints_PLD} 
in form \eqref{eq_state_constraint_theo}. To apply Theorem~\ref{theo_suf_linear}, 
we have to ensure that
\begin{enumerate}
\item $F^0$, $\partial_2 F^0$, $G^0$, $\tilde{A}$ 
and $\tilde{G}$ are continuous for all 
$(t,\xi,\theta)\in[a,a+h]\times\mathbb{R}^{Nn+Nm}$;

\item $F^0(t,\xi)$ is a convex function in $\xi$ for each fixed $t\in[a,a+h]$;

\item $\theta^*$ is a control with response $\xi^*$ that satisfies 
the \emph{maximality condition}
\begin{equation*}
-G^0(t,\theta^*(t))+\Lambda(t)\tilde{G}(t,\theta^*(t))
=\max_{\theta\in\tilde{\Omega}}[-G^0(t,\theta)+\Lambda(t)\tilde{G}(t,\theta)]
\end{equation*}
for almost all $t\in[a,a+h]$. Note that $\tilde{\Omega}\subseteq\mathbb{R}^{Nm}$ 
and $\Lambda(t)$ is any nontrivial solution of the \emph{adjoint system}
\begin{equation*}
\dot{\Lambda}(t)=\partial_2 F^0(t,\xi^*(t))-\Lambda(t)\tilde{A}(t)
\end{equation*}
such that $\Lambda^i(a+h)$ is an inward normal vector of 
the closed convex set 
\begin{equation*}
\begin{split}
\tilde{\Pi}_i=
\begin{cases}
\{\xi_i^*(a+h)\},\ &\text{if }i=0,\ldots,N-2\\
\mathbb{R}^n,\ &\text{if }i=N-1
\end{cases}
\end{split}
\end{equation*}
at the boundary point $\xi_i^*(a+h)$ for $i=0,\ldots,N-1$.
\end{enumerate}
Thus, $\theta^*$ will be an optimal control that leads us 
to the minimal cost $\overline{C}[\theta^*]$.
From now on, we are going to analyze each hypothesis of Theorem~\ref{theo_delay}.
\begin{enumerate}
\item
\begin{enumerate}
\item We have that
\begin{equation*}
\begin{split}
F^0(t,\xi(t))&=\sum_{i=0}^{N-1}f^0(t+hi,\xi_i(t),\xi_{i-k}(t))\\
&=\sum_{i=0}^{N-1}f^0(t+hi,x(t+hi),x(t+h(i-k)))\\
&=\sum_{i=0}^{N-1}f^0(t+hi,x(t+hi),x(t+hi-hk))\\
&=\sum_{i=0}^{N-1}f^0(t+hi,x(t+hi),x(t+hi-r)).
\end{split}
\end{equation*}
By hypothesis, function $f^0$ is continuous with respect to all its arguments. 
Then, $F^0$ is continuous for all $(t,\xi)\in[a,a+h]\times\mathbb{R}^{Nn}$.

\item Having in mind that $N>2k+1$ (see \eqref{hip_h}), 
that is, $k<N-1-k$, then 
\begin{equation*}
\begin{split}
F^0(t,\xi(t))
=& f^0(t_0,\xi_0(t),\xi_{-k}(t))+f^0(t_1,\xi_1(t),\xi_{1-k}(t))\\
&+\ldots+f^0(t_{k-1},\xi_{k-1}(t),\xi_{-1}(t))+f^0(t_k,\xi_k(t),\xi_0(t))\\
&+f^0(t_{k+1},\xi_{k+1}(t),\xi_1(t))\\
&+\ldots+f^0(t_{N-1-k},\xi_{N-1-k}(t),\xi_{N-1-2k}(t))\\
&+\ldots+f^0(t_{N-1},\xi_{N-1}(t),\xi_{N-1-k}(t)).
\end{split}
\end{equation*}
So, for $i=0,\ldots,N-1-k$, we obtain
\begin{equation*}
\begin{split}
\frac{\partial F^0}{\partial \xi_i}(t,\xi(t))
=&\ \partial_2f^0(t+hi,\xi_i(t),\xi_{i-k}(t))
+\partial_3f^0(t+h(k+i),\xi_{k+i}(t),\xi_i(t))\\
=&\ \partial_2f^0(t+hi,x(t+hi),x(t+h(i-k)))\\
&+\partial_3f^0(t+h(k+i),x(t+h(k+i)),x(t+hi))\\
=&\ \partial_2f^0(t+hi,x(t+hi),x(t+hi-r))\\
&+\partial_3f^0(t+hi+r),x(t+hi+r),x(t+hi)).
\end{split}
\end{equation*}
For $i=0,\ldots,N-1-k$ and $t\in[a,a+h]$, we conclude that
$$
a\leq t+hi\leq a+h+h(N-1-k)=b-r.
$$
For $i=N-k,\ldots,N-1$ we have
\begin{equation*}
\begin{split}
\frac{\partial F^0}{\partial \xi_i}(t,\xi(t))
&=\partial_2f^0(t+hi,\xi_i(t),\xi_{i-k}(t))\\
&=\partial_2f^0(t+hi,x(t+hi),x(t+hi-r)).
\end{split}
\end{equation*}
As $i\in\{N-k,\ldots,N-1\}$ and $t\in[a,a+h]$, we obtain
\begin{equation*}
\begin{split}
& a+h(N-k)\leq t+hi\leq a+h+h(N-1)
\Leftrightarrow b-r \leq t+hi \leq b.
\end{split}
\end{equation*}
For each $t\in[a,b]$, there exists $j\in\{0,\ldots,N-1\}$ such that
$$
a+hj\leq t \leq a+h(j+1) 
\Leftrightarrow 
a\leq t-hj \leq a+h.
$$
Thus, let us define $t'\in[a,a+h]$ as being $t'=t-hj$.
Consequently, 
\begin{equation*}
\begin{split}
&\frac{\partial F^0}{\partial \xi_j}(t',\xi(t'))\\
=&\ \partial_2f^0(t'+hj,x(t'+hj),x(t'+hj-r))\\
&+\partial_3f^0(t'+hj+r),x(t'+hj+r),x(t'+hj))\chi(j)_{\{0,\ldots,N-1-k\}}\\
=&\ \partial_2f^0(t,x(t),x(t-r))+\partial_3f^0(t+r,x(t+r),x(t))\chi(t)_{[a,b-r]}.
\end{split}
\end{equation*}
Since $\partial_2f^0$ is continuous for all $(t,x,x_r)\in[a,b]\times\mathbb{R}^{2n}$ 
and function $\partial_3f^0$ is continuous for all $(t,x,x_r)\in[a,b-r]\times\mathbb{R}^{2n}$, 
then $\displaystyle\frac{\partial F^0}{\partial \xi}$ is continuous for all 
$(t,\xi)\in[a,a+h]\times\mathbb{R}^{Nn}$.

\item We have that
\begin{equation*}
\begin{split}
G^0(t,\theta(t))
&=\sum_{i=0}^{N-1}g^0(t+hi,\theta_i(t),\theta_{i-l}(t))\\
&=\sum_{i=0}^{N-1}g^0(t+hi,u(t+hi),u(t+h(i-l)))\\
&=\sum_{i=0}^{N-1}g^0(t+hi,u(t+hi),u(t+hi-hl))\\
&=\sum_{i=0}^{N-1}g^0(t+hi,u(t+hi),u(t+hi-s)).
\end{split}
\end{equation*}
By hypothesis, function $g^0$ is continuous for all $(t,u,u_s)\in[a,b]\times\mathbb{R}^{2m}$. 
Then, $G^0$ is continuous for all $(t,\theta)\in[a,a+h]\times\mathbb{R}^{Nm}$.

\item We know that $\tilde{A}(t)=M(t)+M_D(t)$.
As $A$ and $A_D$ are continuous for all $t\in[a,b]$ and $M(t)$ and $M_D(t)$ 
are depending on $A(t)$ for $t\in[a,b]$ and on $A_D(t)$ for $t\in[a+r,b]$, 
then $\tilde{A}$ is continuous for all $t\in[a,a+h]$.

\item Let us define function $u_s(t)$ by
\begin{equation*}
u_s(t)=u(t-s)
\end{equation*}
for all $t\in[a,b]$. We have already defined
\begin{equation*}
\tilde{G}(t,\theta(t))=M_D^{-}(t)
\left[
\begin{matrix}
\xi^{-}(t)\\
\boldsymbol{0}
\end{matrix}
\right]
+
g_{\theta}(t,\theta(t))+g_{\theta^{-}}(t,\theta^{-}(t)).
\end{equation*}
The matrix $M_D^{-}(t)$ is depending on the matrix $A_D(t)$ 
for $t\in[a,a+r]$. As $A_D(t)$ is continuous in the interval $[a,b]$, then
\begin{equation*}
M_D^{-}(t)\left[
\begin{matrix}
\xi^{-}(t)\\
\boldsymbol{0}
\end{matrix}\right]
\end{equation*}
is continuous for all $t\in[a,a+h]$. Function 
$g_{\theta}(t,\theta(t))+g_{\theta^{-}}(t,\theta^{-}(t))$ 
is continuous if, for each $i=0,\ldots,N-1$, the functions 
$g(t+hi,\theta_i(t))$ and $g_D(t+hi,\theta_{i-l}(t))$ are 
continuous for all $(t,\theta_i(t)),\ (t,\theta_{i-l}(t))
\in[a,a+h]\times\mathbb{R}^m$, respectively. We know that
$g(t+hi,\theta_i(t))=g(t+hi,u(t+hi))$
and
\begin{equation*}
\begin{split}
g_D(t+hi,\theta_{i-l}(t))
&= g_D(t+hi,u(t+h(i-l)))\\
&= g_D(t+hi,u(t+hi-s)),
\end{split}
\end{equation*}
$i=0,\ldots,N-1$. Moreover, as $g(t,u(t))$ and $g_D(t,u_s(t))$ 
are continuous for all $(t,u,u_s)\in[a,b]\times\mathbb{R}^{2m}$,
$\tilde{G}$ is continuous for all 
$(t,\theta)\in[a,a+h]\times\mathbb{R}^{Nm}$.
\end{enumerate}

\item As we know,
\begin{equation*}
F^0(t,\xi(t))=\sum_{i=0}^{N-1}f^0(t+hi,x(t+hi),x(t+hi-r))
\end{equation*}
for $t\in[a,a+h]$ and $f^0$ is convex in $(x,x_r)\in\mathbb{R}^{2n}$ 
for each $t\in[a,b]$. Then, $F^0$ is a convex function in $\xi$ 
for each fixed $t\in[a,a+h]$.

\item If $\theta^*$ is a control with response $\xi^*$ 
that satisfies the \emph{maximality condition}
\begin{equation*}
-G^0(t,\theta^*(t))+\Lambda(t)\tilde{G}(t,\theta^*(t))
=\max_{\theta\in\tilde{\Omega}}[-G^0(t,\theta)+\Lambda(t)\tilde{G}(t,\theta)]
\end{equation*}
for almost all $t\in[a,a+h]$, then
\begin{equation}
\label{eq_cond_max}
-G^0(t,\theta^*(t))+\Lambda(t)\tilde{G}(t,\theta^*(t))
\geq-G^0(t,\theta)+\Lambda(t)\tilde{G}(t,\theta)
\end{equation}
for almost all $t\in[a,a+h]$ and for all admissible $\theta\in\tilde{\Omega}$. 
If we consider that $\eta(t)=\Lambda^j(t-hj)$, then we have that
\begin{equation*}
\Lambda^j(t)
=\Lambda^j(t+hj-hj)
=\eta(t+hj)\Rightarrow\Lambda^j(t')
=\eta(t'+hj)
=\eta(t)
\end{equation*}
and
\begin{equation*}
\begin{split}
\Lambda^{j+l}(t) 
&= \Lambda^{j+l}(t+h(j+l)-h(j+l))\\
&= \Lambda^{j+l}(t+hj+s-h(j+l))\\
&= \eta(t+hj+s),
\end{split}
\end{equation*}
which implies
$\Lambda^{j+l}(t')=\eta(t'+hj+s)=\eta(t+s)$.
As equation~\eqref{eq_cond_max} is verified for all admissible 
$\theta\in\tilde{\Omega}$, we can choose an admissible variable 
$\overline{\theta}\in\tilde{\Omega}$ such that
\begin{equation*}
\begin{split}
\overline{\theta}_i
=
\begin{cases}
u^*(t'+hi),\ & i\neq j\\
u,\ & i=j
\end{cases},
\quad i=0,\ldots,N-1,
\end{split}
\end{equation*}
where $u$ is an admissible control of problem ($LP_D$). 
So, using inequality~\eqref{eq_cond_max} 
and considering $t'_i=t'+hi$, we have that
\begin{equation*}
\begin{split}
&-G^0(t',\theta^*(t'))+\Lambda(t')\tilde{G}(t',\theta^*(t'))
\geq-G^0(t',\overline{\theta})+\Lambda(t')\tilde{G}(t',\overline{\theta})\\
\Leftrightarrow& \sum_{i=0}^{N-1}\{-g^0(t'_i,\theta^*_i(t'),\theta^*_{i-l}(t'))
+\Lambda^i(t')[g(t'_i,\theta^*_i(t'))+g_D(t'_i,\theta^*_{i-l}(t'))]\}\\
&+\sum_{i=0}^{k-1}\Lambda^i(t')A_D(t'_i)\xi_{i-k}(t')\\
&\geq\sum_{i=0}^{N-1}\{-g^0(t'_i,\overline{\theta}_i,\overline{\theta}_{i-l})
+\Lambda^i(t')[g(t'_i,\overline{\theta}_i)+g_D(t'_i,\overline{\theta}_{i-l})]\}\\
&+\sum_{i=0}^{k-1}\Lambda^i(t')A_D(t'_i)\xi_{i-k}(t').
\end{split}
\end{equation*}
As the last sums of both sides of previous inequality are equal, we obtain
\begin{equation}
\label{ineq_max}
\begin{split}
\sum_{i=0}^{N-1}&\{-g^0(t'_i,\theta^*_i(t'),\theta^*_{i-l}(t'))
+\Lambda^i(t')[g(t'_i,\theta^*_i(t'))+g_D(t'_i,\theta^*_{i-l}(t'))]\}\\
&\geq\sum_{i=0}^{N-1}\{-g^0(t'_i,\overline{\theta}_i,\overline{\theta}_{i-l})
+\Lambda^i(t')[g(t'_i,\overline{\theta}_i)+g_D(t'_i,\overline{\theta}_{i-l})]\}.
\end{split} 
\end{equation}
Due to the choice of $\overline{\theta_i}$, $i=0,\ldots, N-1$, some terms of 
the left-hand side of inequality \eqref{ineq_max} cancel with other terms of the
right-hand side. Let us analyze the sums when we only consider the indexes of set 
$I=\{0,\ldots,N-1\}\backslash\{j,j+l\}$. For the first member, we have
\begin{equation*}
\begin{split}
&\sum_{i\in I}\{-g^0(t'_i,\theta^*_i(t'),\theta^*_{i-l}(t'))
+\Lambda^i(t')[g(t'_i,\theta^*_i(t'))+g_D(t'_i,\theta^*_{i-l}(t'))]\}\\
=&\sum_{i\in I}\{-g^0(t'_i,u^*(t'_i),u^*(t'_i-s))\}\\
&+\sum_{i\in I}\{\Lambda^i(t')[g(t'_i,u^*(t'_i))+g_D(t'_i,u^*(t'_i-s))]\}
\end{split}
\end{equation*}
while for the second we obtain
\begin{equation*}
\begin{split}
&\sum_{i\in I}\{-g^0(t'_i,\overline{\theta}_i,\overline{\theta}_{i-l})
+\Lambda^i(t')[g(t'_i,\overline{\theta}_i)+g_D(t'_i,\overline{\theta}_{i-l})]\}\\
=& \sum_{i\in I}\{-g^0(t'_i,u^*(t'_i),u^*(t'_i-s))\}\\
&+\sum_{i\in I}\{\Lambda^i(t')[g(t'_i,u^*(t'_i))+g_D(t'_i,u^*(t'_i-s))]\}.
\end{split}
\end{equation*}
Only the terms associated to the indexes $j, j+l\in \{0,\ldots, N-1\}$ are different. 
Therefore, inequality \eqref{ineq_max} is equivalent to
\begin{equation*}
\begin{split}
\sum_{i\in\{j,j+l\}} & \{-g^0(t'_i,\theta^*_i(t'),\theta^*_{i-l}(t'))
+\Lambda^i(t')[g(t'_i,\theta^*_i(t'))+g_D(t'_i,\theta^*_{i-l}(t'))]\}\\
&\geq\sum_{i\in\{j,j+l\}}\{-g^0(t'_i,\overline{\theta}_i,\overline{\theta}_{i-l})
+\Lambda^i(t')[g(t'_i,\overline{\theta}_i)+g_D(t'_i,\overline{\theta}_{i-l})]\}.
\end{split}
\end{equation*}
For $i=0,\ldots, N-1$, we know that $\overline{\theta}_i=u$, if $i=j$. 
Thus, by the above inequality, it follows that
\begin{eqnarray*}		 			 	
&&-g^0(t'+hj,u^*(t'+hj),u^*(t'+hj-s))\\
&&+\Lambda^j(t')[g(t'+hj,u^*(t'+hj))+g_D(t'+hj,u^*(t'+hj-s))]\\
&&-g^0(t'+hj+s,u^*(t'+hj+s),u^*(t'+hj))\chi_{\{0,\ldots,N-1-l\}}(j)\\
&&+\Lambda^{j+l}(t')[g(t'+hj+s,u^*(t'+hj+s))\\
&&+g_D(t'+hj+s,u^*(t'+hj))]\chi_{\{0,\ldots,N-1-l\}}(j)\\
&&\geq-g^0(t'+hj,u,u^*(t'+hj-s))\\
&&+\Lambda^j(t')[g(t'+hj,u)+g_D(t'+hj,u^*(t'+hj-s))]\\
&&-g^0(t'+hj+s,u^*(t'+hj+s),u)\chi_{\{0,\ldots,N-1-l\}}(j)\\
&&+\Lambda^{j+l}(t')[g(t'+hj+s,u^*(t'+hj+s))\\
&&+g_D(t'+hj+s,u)]\chi_{\{0,\ldots,N-1-l\}}(j).
\end{eqnarray*}
As $t'=t-hj\in[a,a+h]$ and $0\leq j \leq N-1-l$, then
\begin{equation*}
\begin{split}
0\leq hj \leq Nh-h-s
&\Leftrightarrow \ a\leq t'+hj \leq a+h+Nh-h-s\\
&\Leftrightarrow \ a\leq t'+hj\leq b-s.
\end{split}
\end{equation*}
Consequently, we have that
\begin{equation*}
\begin{split}
&-g^0(t,u^*(t),u^*(t-s))+\Lambda^j(t')[g(t,u^*(t))+g_D(t,u^*(t-s))]\\
&-g^0(t+s,u^*(t+s),u^*(t))\chi_{[a,b-s]}(t)\\
&+\Lambda^{j+l}(t')[g(t+s,u^*(t+s))+g_D(t+s,u^*(t))]\chi_{[a,b-s]}(t)\\
&\geq-g^0(t,u,u^*(t-s))+\Lambda^j(t')[g(t,u)+g_D(t,u^*(t-s))]\\
&-g^0(t+s,u^*(t+s),u)\chi_{[a,b-s]}(t)\\
&+\Lambda^{j+l}(t')[g(t+s,u^*(t+s))+g_D(t+s,u)]\chi_{[a,b-s]}(t).
\end{split}
\end{equation*}
As some terms cancel, we obtain
\begin{equation*}
\begin{split}
&-g^0(t,u^*(t),u^*(t-s))+\Lambda^j(t')g(t,u^*(t))\\
&-g^0(t+s,u^*(t+s),u^*(t))\chi_{[a,b-s]}(t)+\Lambda^{j+l}(t')g_D(t+s,u^*(t))\chi_{[a,b-s]}(t)\\
&\geq-g^0(t,u,u^*(t-s))+\Lambda^j(t')g(t,u)\\
&-g^0(t+s,u^*(t+s),u)\chi_{[a,b-s]}(t)+\Lambda^{j+l}(t')g_D(t+s,u)\chi_{[a,b-s]}(t).		 			 	
\end{split}
\end{equation*}
Using relations $\Lambda^j(t')=\eta(t)$ and $\Lambda^{j+l}(t')=\eta(t+s)$, we have that
\begin{equation}
\label{eq0_theodelay}
\begin{split}
&-g^0(t,u^*(t),u^*(t-s))+\eta(t)g(t,u^*(t))\\
&+[-g^0(t+s,u^*(t+s),u^*(t))+\eta(t+s)g_D(t+s,u^*(t))]\chi_{[a,b-s]}(t)\\
&\geq-g^0(t,u,u^*(t-s))+\eta(t)g(t,u)\\
&+[-g^0(t+s,u^*(t+s),u)+\eta(t+s)g_D(t+s,u)]\chi_{[a,b-s]}(t).
\end{split}	
\end{equation}
Attending to the definition of $H_D^p$, $p\in\{0,1\}$, the inequality 
\eqref{eq0_theodelay} is equivalent to the \emph{maximality condition} 
\eqref{max_cond_delayed_theo} of Theorem~\ref{theo_delay}.
Furthermore, we cannot forget that $\Lambda(t)$ 
is any nontrivial solution of the \emph{adjoint system}
\begin{equation}
\label{adj_syst_auxprob}
\dot{\Lambda}(t)=\partial_2F^0(t,\xi^*(t))-\Lambda(t)\tilde{A}(t)
\end{equation}
that satisfies the \emph{transversality condition} 
(see Remark~\ref{remark_transv_cond})
\begin{equation}
\label{trans_cond_auxprob}
\Lambda^{N-1}(a+h)=[
\begin{matrix}
0 & \cdots & 0
\end{matrix}
]_{1\times n}.
\end{equation}
As we know,
\begin{equation*}
\tilde{A}(t)= M(t)+M_D(t)
\end{equation*}
and
$
\Lambda(t)=
[
\begin{matrix}
\Lambda^0(t) & \Lambda^1(t) & \cdots & \Lambda^{N-1}(t)
\end{matrix}
],$
where $\Lambda^i(t)$ has dimension $1\times n$ for all $i\in\{0,\ldots,N-1\}$.
Consequently, by the adjoint system \eqref{adj_syst_auxprob}, we can write that
\begin{equation*}
\begin{split}
\dot{\Lambda}^i(t)=&\ \partial_2f^0(t+hi,\xi^*_i(t),\xi^*_{i-k}(t))\\
&+\partial_3f^0(t+h(i+k),\xi^*_{k+i}(t),\xi^*_i(t))
\chi_{\{0,\ldots,N-1-k\}}(i)-\Lambda^i(t)A(t+hi)\\
&-\Lambda^{i+k}(t)A_D(t+h(i+k))\chi_{\{0,\ldots,N-1-k\}}(i)\\
=&\ \partial_2f^0(t+hi,x^*(t+hi),x^*(t+hi-hk))\\
&+\partial_3f^0(t+hi+hk,x^*(t+hi+hk),x^*(t+hi))\chi_{\{0,\ldots,N-1-k\}}(i)\\
&-\Lambda^i(t)A(t+hi)-\Lambda^{i+k}(t)A_D(t+hi+hk)\chi_{\{0,\ldots,N-1-k\}}(i)\\
=&\ \partial_2f^0(t+hi,x^*(t+hi),x^*(t+hi-r))\\
&+\partial_3f^0(t+hi+r,x^*(t+hi+r),x^*(t+hi))\chi_{\{0,\ldots,N-1-k\}}(i)\\
&-\Lambda^i(t)A(t+hi)-\Lambda^{i+k}(t)A_D(t+hi+r)\chi_{\{0,\ldots,N-1-k\}}(i).
\end{split}
\end{equation*}
Furthermore, as $\eta(t)=\Lambda^j(t-hj)$, we conclude that
\begin{equation}
\label{eq1_theodelay}
\begin{split}
\dot{\eta}(t)=&\ \dot{\Lambda}^j(t-hj)\\
=&\ \partial_2f^0(t,x^*(t),x^*(t-r))+\partial_3
f^0(t+r,x^*(t+r),x^*(t))\chi_{[a,b-r]}(t)\\
&-\eta(t)A(t)-\eta(t+r)A_D(t+r)\chi_{[a,b-r]}(t).
\end{split}
\end{equation}
By equation \eqref{trans_cond_auxprob},
\begin{equation*}
\begin{split}
\Lambda^{N-1}(a+h)=
[
\begin{matrix}
0 & \cdots & 0
\end{matrix}
]_{1\times n}
\Leftrightarrow\ &\eta(a+h+h(N-1))=
[
\begin{matrix}
0 & \cdots & 0
\end{matrix}
]_{1\times n}\\
\Leftrightarrow\ &\eta(a+hN)=
[
\begin{matrix}
0 & \cdots & 0
\end{matrix}
]_{1\times n}.
\end{split}
\end{equation*}
As $a+hN=b$, we obtain the \emph{transversality condition}
\begin{equation}
\label{eq2_theodelay}
\eta(b)=
[
\begin{matrix}
0 & \cdots & 0
\end{matrix}
]_{1\times n}.
\end{equation}
With conditions \eqref{eq0_theodelay}, 
\eqref{eq1_theodelay} and \eqref{eq2_theodelay}, 
we obtain item $3$ of Theorem~\ref{theo_delay}. 
\end{enumerate}
The proof is complete.
\end{proof}

\begin{remark}
\label{remark_transv_cond}
We can note that:
(i) problems ($LP_D$) and ($\overline{LP}$) are equivalent;
(ii) the augmented and non-delayed problem ($\overline{LP}$) is defined for $t\in[a,a+h]$.
Even more, we can solve problem ($LP_D$) by solving $N$ sub-problems, each one with respect 
to each subinterval of $[a,b]$ with amplitude $h$. Then, we can concatenate the respective $N$ 
optimal solutions in order to obtain an optimal solution of ($LP_D$). Thus, we can solve 
problem ($LP_D$) by solving $N$ augmented and non-delayed sub-problems ($\overline{LP}_i$) 
associated to problem ($LP_D$), with $i=0,\ldots,N-1$. For $i\in\{0,\ldots,N-2\}$, the $(i+1)$th 
augmented and non-delayed sub-problem ($\overline{LP}_i$) consists to minimize
\begin{equation*}
\int\limits_{a}^{a+h}f^0(t_i,\xi_i(t),\xi_{i-k}(t))+g^0(t_i,\theta_i(t),\theta_{i-l}(t))dt
\end{equation*}
subject to
\begin{equation*}
\begin{split}
&\dot{\xi}_i(t)=\ A(t_i)\xi_i(t)+A_D(t_i)\xi_{i-k}(t)
+g(t_i,\theta_i(t))+g_D(t_i,\theta_{i-l}(t))\\
&\xi_i(a)=
\begin{cases}
\varphi(a), &\text{ if } i=0\\
\ \xi_{i-1}(a+h), &\text{ if } i=1,\ldots,N-2
\end{cases}\\
&\xi_i(a+h)\in\tilde{\Pi}_i=\{\xi_i^*(a+h)\}
\end{split}
\end{equation*}
for $t\in[a,a+h]$. Theorem~\ref{theo_suf_linear} can be applied and we can find 
an optimal pair $(\xi_i^*(\cdot),\theta_i^*(\cdot))$ in the interval of time $[a,a+h]$ 
that provides an optimal solution $(x^*(\cdot),u^*(\cdot))$ in the interval of time 
$[a+hi,a+h(i+1)]$. The set $\tilde{\Pi}_i$ has a single point. So, $\Lambda^i(a+h)$ 
is an inward normal vector of $\tilde{\Pi}_i$ at the boundary point $\xi_i^*(a+h)$
(recall Remark~\ref{vacuous}). The last augmented and non-delayed sub-problem 
($\overline{LP}_{N-1}$) consists to minimize
\begin{equation*}
\int\limits_{a}^{a+h}f^0(t_{N-1},\xi_{N-1}(t),
\xi_{N-1-k}(t))+g^0(t_{N-1},\theta_{N-1}(t),\theta_{N-1-l}(t))dt
\end{equation*}
subject to
\begin{equation*}
\begin{split}
&\dot{\xi}_{N-1}(t)= A(t_{N-1})\xi_{N-1}(t)
+A_D(t_{N-1})\xi_{N-1-k}(t)+g(t_{N-1},\theta_{N-1}(t))\\
&\ \ \ \ \ \ \ \ \ \ \ \ \ +g_D(t_{N-1},\theta_{N-1-l}(t))\\
&\xi_{N-1}(a)=\ \xi_{N-2}(a+h)\\
&\xi_{N-1}(a+h)\in\tilde{\Pi}_{N-1}=\mathbb{R}^n
\end{split}
\end{equation*}
for $t\in[a,a+h]$. Again, Theorem~\ref{theo_suf_linear} can be applied and we can find 
an optimal pair $(\xi_{N-1}^*(\cdot),\theta_{N-1}^*(\cdot))$ in interval of time $[a,a+h]$ 
that provides an optimal solution $(x^*(\cdot),u^*(\cdot))$ in the interval of time $[a+h(N-1),b]$. 
As $\tilde{\Pi}_{N-1}=\mathbb{R}^n$, then by Theorem~\ref{theo_suf_linear}
$\Lambda^{N-1}(a+h)
=[\begin{matrix}
0 & \cdots & 0
\end{matrix}
]_{1\times n}$.
\end{remark}

% -------------------------------

\section{An illustrative example}
\label{sect_example}

Let us consider problem $(P)$ given by
\begin{equation}
\label{example}
\tag{$P$}
\begin{split}
\min\ \ \ & C[u]=\int_{0}^{4}x(t)+100u^2(t)dt\\
\text{s.t.}\ \ & \dot{x}(t)=x(t)+x(t-2)-10u(t-1),\\
& x(t)=1,\ t\in[-2,0],\\
& u(t)=0,\ t\in[-1,0[,
\end{split}
\end{equation}
where $u(t)\in\Omega=\mathbb{R}$ for each $t\in[-1,4]$. Thus, 
we have that $n=m=1$, $a=0$, $b=4$, $r=2$, $s=1$, $f^0(t,x(t),x(t-2))=x(t)$, 
$g^0(t,u(t),u(t-1))=100u^2(t)$, $A(t)=A_D(t)=1$, $g(t,u(t))=0$ 
and $g_D(t,u(t-1))=-10u(t-1)$. Note that our functions respect 
hypothesis $1$ and $2$ of Theorem~\ref{theo_delay}. Let $\bar{u}$ 
be an admissible control of problem \eqref{example} and let us maximize function
\begin{equation*}
\begin{split}
&-g^0(t,u,\bar{u}(t-1))+\eta(t)g(t,u)\\
&+[-g^0(t+1,\bar{u}(t+1),u)+\eta(t+1)g_D(t+1,u)]\chi_{[0,3]}(t)\\
=&-100u^2+[-100\bar{u}^2(t+1)-10\eta(t+1)u]\chi_{[0,3]}(t)\\
=&
\begin{cases}
-100u^2-10\eta(t+1)u-100\bar{u}^2(t+1),\ & t\in[0,3]\\
-100u^2,\ & t\in\ ]3,4]
\end{cases}
\end{split}
\end{equation*}
with respect to $u\in\mathbb{R}$. We obtain 
\begin{equation*}
u(t)=-\frac{\eta(t+1)}{20}
\end{equation*}
for $t\in[0,3]$ and $u(t)=0$ for $t\in\ ]3,4]$. 
Furthermore, we know that $\eta(t)$ is any nontrivial solution of
\begin{equation*}
\begin{split}
\dot{\eta}(t)=&\ \partial_2f^0(t,x(t),x(t-2))+\partial_3
f^0(t+2,x(t+2),x(t))\chi_{[0,2]}(t)-\eta(t)A(t)\\
&-\eta(t+2)A_D(t+2)\chi_{[0,2]}(t)\\
\Leftrightarrow
\dot{\eta}(t)=&\ 1-\eta(t)-\eta(t+2)\chi_{[0,2]}(t)
=
\begin{cases}
1-\eta(t)-\eta(t+2),\ & t\in[0,2]\\
1-\eta(t),\ & t\in\ ]2,4]
\end{cases}
\end{split}
\end{equation*}
that satisfies the \emph{transversality condition} $\eta(4)=0$. 
The adjoint system is given by
\begin{equation}
\begin{split}
\label{adjoint_system}
\begin{cases}
\dot{\eta}(t)=
\begin{cases}
1-\eta(t)-\eta(t+2),\ & t\in[0,2]\\
1-\eta(t),\ & t\in\ ]2,4]\\
\end{cases}\\
\eta(4)=0.
\end{cases}
\end{split}
\end{equation}
For $t\in\ ]2,4]$, the solution of differential equation
\begin{equation*}
\begin{cases}
\dot{\eta}(t)
= 1-\eta(t)\\
\eta(4)=0
\end{cases}
\end{equation*}
is given by
\begin{equation*}
\eta(t)=1-e^{4-t}.
\end{equation*}
Knowing $\eta(t)$, $t\in\ ]2,4]$, and attending to the continuity 
of function $\eta$ for all $t\in[0,4]$, we can determine $\eta(t)$ 
for $t\in[0,2]$ solving the differential equation
\begin{equation*}
\begin{cases}
\dot{\eta}(t)
=1-\eta(t)-\eta(t+2)\\
\eta(2)=1-e^{4-2}=1-e^2
\end{cases}
\end{equation*}
for $t\in[0,2]$. Therefore,
\begin{equation*}
\eta(t)=e^{2-t}(t-e^2-1),\ t\in[0,2],
\end{equation*}
and, consequently, the solution of the adjoint system 
\eqref{adjoint_system} is given by
\begin{equation*}
\begin{split}
\eta(t)
=
\begin{cases}
e^{2-t}(t-e^2-1),\ & t\in [0,2]\\
1-e^{4-t},\ & t\in\ ]2,4].
\end{cases}
\end{split}
\end{equation*}
So, the control is given by
\begin{equation}
\begin{split}
\label{optimalcontrol}
u(t)=\frac{1}{20}
\begin{cases}
0,\ & t\in[-1,0[\\
e^{3-t}-e^{1-t}t,\ & t\in[0,1[\\
e^{3-t}-1,\ & t\in[1,3]\\
0,\ & t\in\ ]3,4].
\end{cases}
\end{split}
\end{equation}
Knowing the control, we can determine 
the state by solving the differential equation
\begin{equation*}
\begin{cases}
\dot{x}(t)=x(t)+x(t-2)-10u(t-1)\\
x(t)=1,\ t\in[-2,0].
\end{cases}
\end{equation*}
The state solution is
\begin{equation}
\label{optimalstate}
\begin{split}
x(t)=
\begin{cases}
1,\ & t\in[-2,0]\\[0.8em]
-1+2e^t,\ & t\in\ ]0,1]\\[0.8em]
\displaystyle\frac{(e^2+2e^4-2e^2t)e^{-t}-8+(17-2e^2)e^t}{8},\ & t\in\ ]1,2]\\[0.8em]
\displaystyle\frac{2e^{4-t}+4+(-47e^{-2}+17-2e^2+16e^{-2}t)e^t}{8},\ & t\in\ ]2,3]\\[0.8em]
\displaystyle\frac{(-e^6+e^4t)e^{-t}+4+(-51e^{-2}+24-2e^2+17e^{-2}t-2t)e^t}{8},\ & t\in\ ]3,4].
\end{cases}
\end{split}
\end{equation}
Such analytical expressions can be obtained with the help of a modern computer 
algebra system. We have used \textsf{Mathematica}.
In Figures~\ref{FigControlo} and \ref{FigState}, we observe that the numerical 
solutions for control and state, obtained using \textsf{AMPL} \cite{AMPL} and 
\textsf{IPOPT} \cite{IPOPT}, are in agreement with their analytical solutions, given by 
\eqref{optimalcontrol} and \eqref{optimalstate}, respectively. The numerical solutions 
were obtained using Euler's forward difference method in \textsf{AMPL} and \textsf{IPOPT}, 
dividing the interval of time $[0,4]$ into 2000 subintervals. The minimal cost is 
$$
\frac{23+e^2+34e^4-2e^6}{16}\simeq 67.491786.
$$
% -------------------------------
\begin{figure}[htp]
\begin{center}
\includegraphics[scale=0.3]{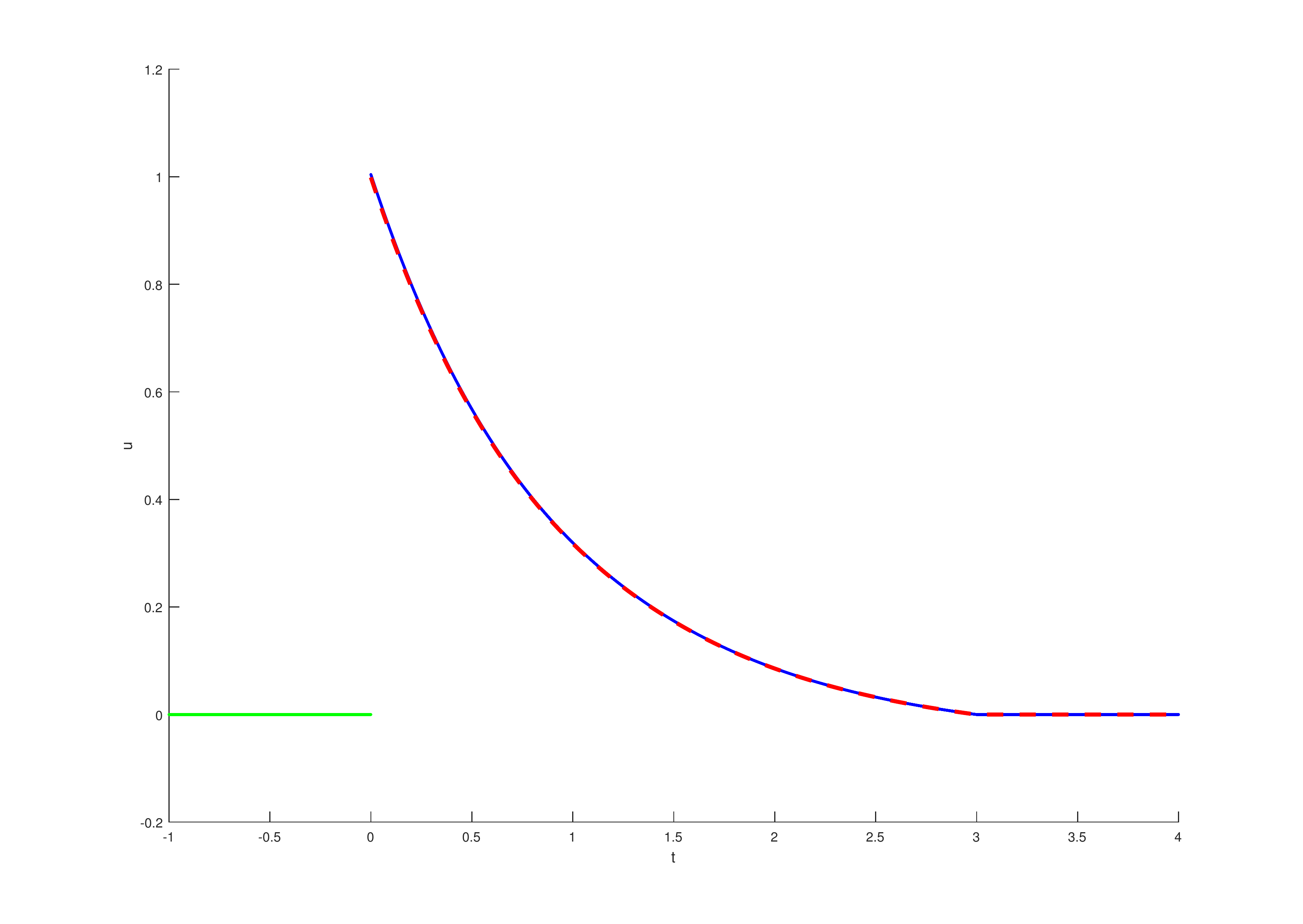}
\caption{Optimal control: green line -- initial data; 
blue line -- analytical solution; 
red dashed line -- numerical solution.}
\label{FigControlo}
\end{center}
\end{figure}
% -------------------------------
\begin{figure}[htp]
\begin{center}
\includegraphics[scale=0.3]{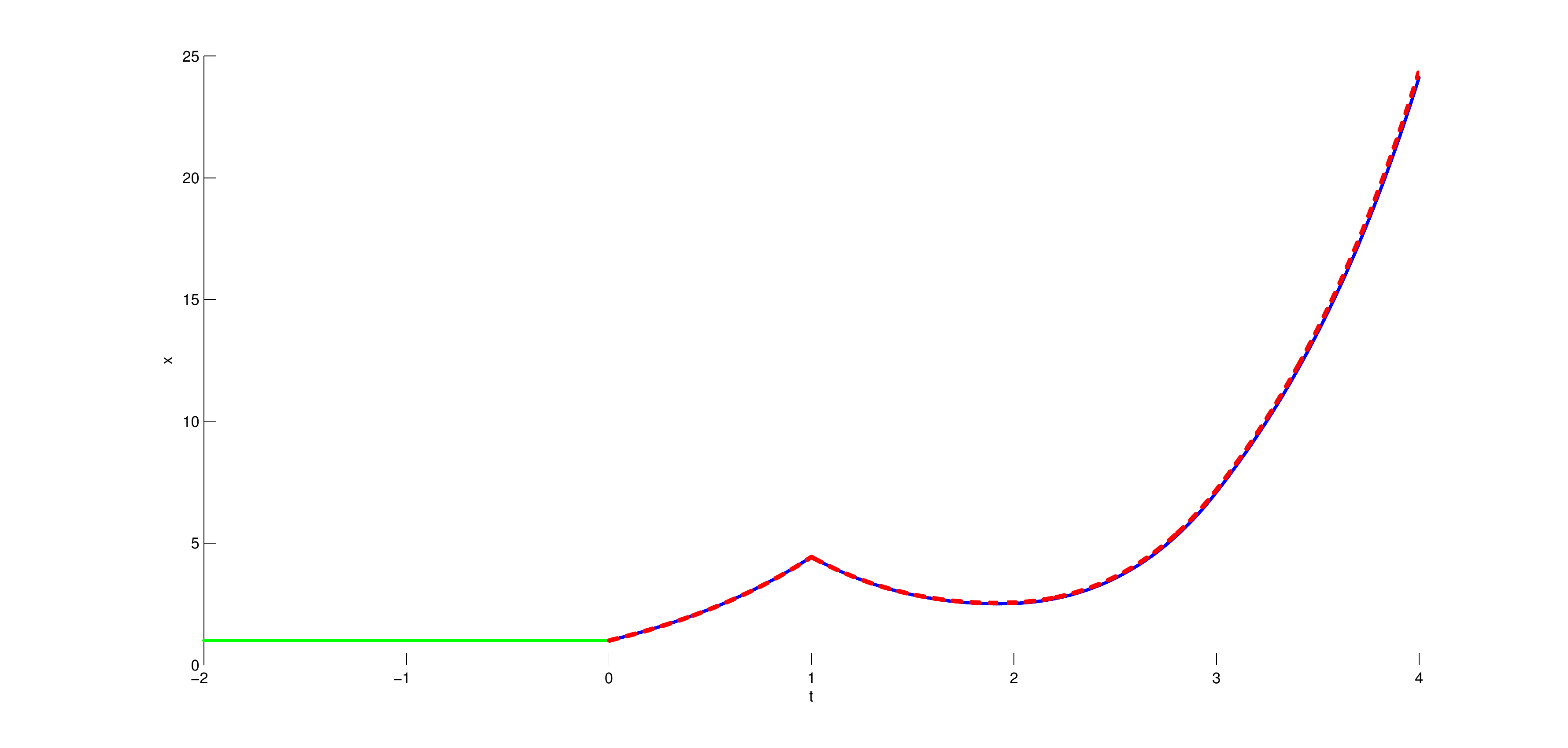}
\caption{Optimal state: green line -- initial data; 
blue line -- analytical solution; 
red dashed line -- numerical solution.}
\label{FigState}
\end{center}
\end{figure}

% -------------------------------

\section{Conclusion}
\label{sec:conclusion}

We considered a delayed state-linear optimal control problem. 
We proved a sufficient optimality condition for
problems with delays in both state and control variables. 
The proof is based on the transformation of the delayed 
state-linear optimal control problem into a non-delayed one, 
following the approach proposed in \cite{Guinn} and used 
in \cite{Gollmann}. Analogously to \cite{Gollmann}, we ensure 
the commensurability assumption between the, possibly different, 
state and control delays. An example is provided, which illustrates 
the usefulness of obtained sufficient optimality condition.

% -------------------------------

\section*{Acknowledgements}

This research was supported by the
Portuguese Foundation for Science and Technology (FCT)
within projects UID/MAT/04106/2019 (CIDMA)
and PTDC/EEI-AUT/2933/2014 (TOCCATA), funded by Project
3599 -- Promover a Produ\c{c}\~ao Cient\'{\i}fica e Desenvolvimento
Tecnol\'ogico e a Constitui\c{c}\~ao de Redes Tem\'aticas
and FEDER funds through COMPETE 2020, Programa Operacional
Competitividade e Internacionaliza\c{c}\~ao (POCI).
Lemos-Pai\~{a}o is also supported by the Ph.D.
fellowship PD/BD/114184/2016; Silva by national funds (OE), 
through FCT, I.P., in the scope of the framework contract foreseen 
in the numbers 4, 5 and 6 of the article 23, of the Decree-Law 57/2016, 
of August 29, changed by Law 57/2017, of July 19. The authors
are very grateful to a referee for carefully reading of their 
manuscript and for several constructive remarks.

% -------------------------------

% -------------------------------

\medskip

% -------------------------------

Received December 2017; revised April 2018; accepted January 2019.

% -------------------------------

\medskip

% -------------------------------

\end{document}